\documentclass[12pt]{article}
\usepackage{amsmath, amsthm, amssymb}
\usepackage{tkz-graph}

\theoremstyle{plain}
\newtheorem{thm}{Theorem}

\newtheorem{lem}[thm]{Lemma}

\newtheorem*{KKT}{Kierstead and Kostochka}
\newtheorem*{TheoremM}{Theorem M}
\newtheorem*{CorN}{Corollary N}
\newtheorem*{CorO}{Corollary O}

\theoremstyle{definition}
\newtheorem{defn}{Definition}
\theoremstyle{remark}

\title{Coloring $\Delta$-Critical Graphs With Small High Vertex Cliques}
\author{Landon Rabern\\
\small \texttt{landon.rabern@gmail.com}}
\begin{document}
\maketitle
\begin{abstract}
We prove that $K_{\chi(G)}$ is the only critical graph $G$ with $\chi(G) \geq \Delta(G) \geq 6$ and $\omega(\mathcal{H}(G)) \leq \left \lfloor \frac{\Delta(G)}{2} \right \rfloor - 2$.  Here $\mathcal{H}(G)$ is the subgraph of $G$ induced on the vertices of degree at least $\chi(G)$.  Setting $\omega(\mathcal{H}(G)) = 1$ proves a conjecture of Kierstead and Kostochka.
\end{abstract}

\section{Introduction}
Given a graph $G$, let $\mathcal{H}(G)$ be the subgraph of $G$ induced on the vertices of degree at least $\chi(G)$.  Recently, Kierstead and Kostochka \cite{KK} proved the following theorem and conjectured that the $7$ could be improved to $6$.

\begin{KKT}
$K_{\chi(G)}$ is the only critical graph $G$ with $\chi(G) \geq \Delta(G) \geq 7$ such that $\mathcal{H}(G)$ is independent.
\end{KKT}

We prove this conjecture by establishing the following generalization.

\begin{TheoremM}
$K_{\chi(G)}$ is the only critical graph $G$ with $\chi(G) \geq \Delta(G) \geq 6$ and $\omega(\mathcal{H}(G)) \leq \left \lfloor \frac{\Delta(G)}{2} \right \rfloor - 2$.
\end{TheoremM}

Setting $\omega(\mathcal{H}(G)) = 1$ proves the conjecture.

\begin{CorN}
$K_{\chi(G)}$ is the only critical graph $G$ with $\chi(G) \geq \Delta(G) \geq 6$ such that $\mathcal{H}(G)$ is independent.
\end{CorN}

We can restate this in terms of Ore-degree as in \cite{KK} to get a generalization of Brooks' theorem.

\begin{defn}
The \emph{Ore-degree} of an edge $xy$ in a graph $G$ is $\theta(xy) = d(x) + d(y)$.  The \emph{Ore-degree} of a graph $G$ is $\theta(G) = \max_{xy \in E(G)}\theta(xy)$.
\end{defn}

\begin{CorO}
If $6 \leq \chi(G) = \left \lfloor \frac{\theta(G)}{2} \right \rfloor + 1$, then $G$ contains the complete graph $K_{\chi(G)}$.
\end{CorO}

This is best possible as shown by the following example from \cite{KK}.

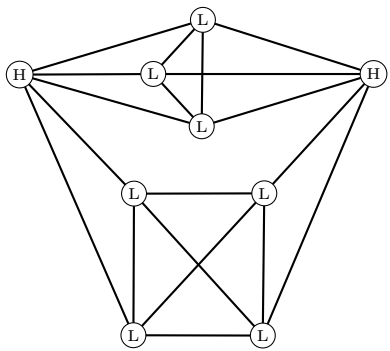
\begin{figure}[h]
\centering
\begin{tikzpicture}[scale = 10]
\tikzstyle{VertexStyle}=[shape = circle,	
								 minimum size = 1pt,
								 inner sep = 1.2pt,
                         draw]
\Vertex[x = 0.270266681909561, y = 0.890800006687641, L = \tiny {L}]{v0}
\Vertex[x = 0.336266696453094, y = 0.962799992412329, L = \tiny {L}]{v1}
\Vertex[x = 0.334666579961777, y = 0.821199983358383, L = \tiny {L}]{v2}
\Vertex[x = 0.56306654214859, y = 0.890800006687641, L = \tiny {H}]{v3}
\Vertex[x = 0.244666695594788, y = 0.731600046157837, L = \tiny {L}]{v4}
\Vertex[x = 0.417866677045822, y = 0.732000052928925, L = \tiny {L}]{v5}
\Vertex[x = 0.243866696953773, y = 0.543200016021729, L = \tiny {L}]{v6}
\Vertex[x = 0.415866762399673, y = 0.542800068855286, L = \tiny {L}]{v7}
\Vertex[x = 0.0926666706800461, y = 0.890000000596046, L = \tiny {H}]{v8}
\tikzstyle{EdgeStyle}=[]
\Edge[](v1)(v0)
\tikzstyle{EdgeStyle}=[]
\Edge[](v2)(v0)
\tikzstyle{EdgeStyle}=[]
\Edge[](v3)(v0)
\tikzstyle{EdgeStyle}=[]
\Edge[](v2)(v1)
\tikzstyle{EdgeStyle}=[]
\Edge[](v3)(v1)
\tikzstyle{EdgeStyle}=[]
\Edge[](v2)(v3)
\tikzstyle{EdgeStyle}=[]
\Edge[](v5)(v4)
\tikzstyle{EdgeStyle}=[]
\Edge[](v6)(v4)
\tikzstyle{EdgeStyle}=[]
\Edge[](v6)(v5)
\tikzstyle{EdgeStyle}=[]
\Edge[](v6)(v7)
\tikzstyle{EdgeStyle}=[]
\Edge[](v7)(v4)
\tikzstyle{EdgeStyle}=[]
\Edge[](v7)(v5)
\tikzstyle{EdgeStyle}=[]
\Edge[](v4)(v8)
\tikzstyle{EdgeStyle}=[]
\Edge[](v6)(v8)
\tikzstyle{EdgeStyle}=[]
\Edge[](v5)(v3)
\tikzstyle{EdgeStyle}=[]
\Edge[](v7)(v3)
\tikzstyle{EdgeStyle}=[]
\Edge[](v0)(v8)
\tikzstyle{EdgeStyle}=[]
\Edge[](v1)(v8)
\tikzstyle{EdgeStyle}=[]
\Edge[](v2)(v8)
\end{tikzpicture}

\caption{A counterexample to Corollary O with $\chi = 5$.}
\end{figure}

\section{The Proof}
We will use part of an algorithm of Mozhan \cite{Mozhan}.  The following is a generalization of his main lemma.

\begin{lem}\label{mozlemma}
Let $G$ be a graph containing at least one critical vertex.  Of all $\chi(G)$-colorings of $G$ of the form

\[\left\{\{x\}, L_{11}, L_{12}, \ldots, L_{1r_1}, L_{21}, L_{22}, \ldots, L_{2r_2}, \ldots, L_{a1}, L_{a2}, \ldots, L_{ar_a}\right\}\]

pick one (call it $\pi$) minimizing

\[\sum_{i = 1}^a \left|E\left(G\left[\bigcup_{j = 1}^{r_i} L_{ij}\right]\right)\right|.\]

Put $U_i = \bigcup_{j = 1}^{r_i} L_{ij}$ and let $Z_i(x)$ be the component of $x$ in $G[\{x\} \cup U_i]$.  If $d_{Z_i(x)}(x) = r_i$, then $Z_i(x)$ is complete if $r_i \geq 3$ and $Z_i(x)$ is an odd cycle if $r_i = 2$.
\end{lem}
\begin{proof}
Let $1 \leq i \leq r_i$ such that $d_{Z_i(x)}(x) = r_i$. Put $Z_i = Z_i(x)$.\newline

First assume that $\Delta(Z_i) > r_i$.  Take $y \in V(Z_i)$ with $d_{Z_i}(y) > r_i$ closest to $x$ and let $x_1x_2\cdots x_t$ be a shortest $x-y$ path in $Z_i$.  Plainly, for $k < t$, each $x_k$ hits exactly one vertex in each color class besides its own.  Thus we may recolor $x_k$ with $\pi(x_{k + 1})$ for $k < t$ and $x_t$ with $\pi(x_1)$ to produce a new $\chi(G)$-coloring of $G$ (this can be seen as a generalized Kempe chain).  But we've moved a vertex ($x_t$) of degree $r_i + 1$ out of $U_i$ while moving in a vertex ($x_1$) of degree $r_i$ violating the minimality condition on $\pi$.  This is a contradiction.\newline

Thus $\Delta(Z_i) \leq r_i$.  But $\chi(Z_i) = r_i + 1$, so Brooks' theorem implies that $Z_i$ is complete if $r_i \geq 3$ and $Z_i$ is an odd cycle if $r_i = 2$.
\end{proof}

Now to prove Theorem M, we assume it is false and derive a contradiction from properties of a minimal counterexample.  Let $G \neq K_{\chi(G)}$ be a critical graph with $\chi(G) \geq \Delta(G) \geq 6$ and $\omega(\mathcal{H}(G)) \leq \left \lfloor \frac{\Delta(G)}{2} \right \rfloor - 2$ having the minimum number of vertices.

\begin{defn}
We call $v \in V(G)$ \emph{low} if $d(v) = \chi(G) - 1$ and \emph{high} otherwise.
\end{defn}

\begin{lem}
If $\Delta(G) = 6$, then $G$ contains no $K_{6} - e$.
\end{lem}
\begin{proof}
Assume $\Delta(G) = 6$ and that $G$ contains a $K_{6} - e$, call it $H$.  Let $x_1, x_2 \in V(H)$ with $d_H(x_i) = 4$.  Color $G - H$ with $5$ colors and let $J$ be the resulting list assignment on $H$.  Then
$|J(x_1)| + |J(x_2)| \geq d_H(x_1) + d_H(x_2) - 2 \geq 2*6 - 6 \geq 6$.  Hence we have $c \in J(x_1) \cap J(x_2)$.  Color both $x_1$ and $x_2$ with $c$ to get a list assignment $J'$ on $F = H - \{x_1,x_2\}$. Since $\Delta(G) = 6$, $\mathcal{H}(G)$ is independent.  Thus at most one vertex $y \in V(F)$ is high.  Hence $|J'(y)| \geq 3$ and $|J'(z)| \geq 4$ for all $z \in V(F) - \{y\}$.  Since $F$ has $4$ vertices we can complete the $5$ coloring using Hall's theorem.  This contradiction completes the proof.
\end{proof}

\begin{lem}\label{NoTwoHitter}
Assume $\Delta(G) = 6$. Let $C$ be a $K_{5}$ clique in $G$ with at most one high vertex.  Then each vertex in $G - C$ is adjacent to at most one low vertex in $C$.
\end{lem}

\begin{proof}
Assume otherwise that some $x \in V(G - C)$ is adjacent to all of $S \subseteq C$ where each vertex in $S$ is low and $|S| \geq 2$. Put $F = G - C$.  Then $F$ is $5$ colorable.  Since each vertex in $C$ is adjacent to at least one vertex in $F$ and $G$ contains no $K_{6} - e$, we have $y \in V(F)$  with $y \neq x$ such that $N(y) \cap C$ contains low vertices.  Consider the graph $T = F + xy$.  Note that $d_T(x) \leq 5$ and $d_T(y) \leq 6$.  By minimality of $G$, $T$ is either $5$ colorable or contains a $K_{\Delta(G)}$.  In the former case we get a $5$ coloring of $F$ where $x$ and $y$ receive different colors, but this is easily completable to a coloring of $G$.  Thus $T$ contains $K_{6}$ and hence $G$ contains a $K_{6} - e$ giving a contradiction.
\end{proof}

Note that in Lemma \ref{mozlemma}, if $d_{Z_i(x)}(x) = r_i$ then we can \emph{swap} $x$ with any other $y \in Z_i(x)$ by changing $\pi$ so that $x$ is colored with $\pi(y)$ and $y$ is colored with $\pi(x)$ to get another minimal $\chi(G)$-coloring of $G$. \newline

\begin{proof}[Proof of Theorem M]
First, if $\chi(G) > \Delta(G)$ the theorem follows from Brooks' theorem.\newline

Hence we may assume that $\chi(G) = \Delta(G)$.  Put $\Delta = \Delta(G)$, $r_1 = \left \lfloor \frac{\Delta - 1}{2} \right \rfloor$ and $r_2 = \left \lceil \frac{\Delta - 1}{2} \right \rceil$.  Of all $\chi(G)$ colorings of $G$ of the form

\[\left\{\{x\}, L_{11}, \ldots, L_{1r_1}, L_{21}, \ldots, L_{2r_2} \right \}\]

pick one minimizing

\[\sum_{i = 1}^2 \left|E\left(G\left[\bigcup_{j = 1}^{r_i} L_{ij}\right]\right)\right|.\]

Throughout the proof we refer to a coloring that minimizes the above function as a \emph{minimal} coloring. Put $U_i = \bigcup_{j = 1}^{r_i} L_{ij}$ and let $C_i = \pi(U_i)$ (the colors used on $U_i$).  For a minimal coloring $\gamma$ of $G$, let $Z_{\gamma, i}(x)$ be the component of $x$ in $G[\{x\} \cup \gamma^{-1}(C_i)]$.  Put $Z_i(x) = Z_{\pi, i}(x)$.\newline

Note that $r_1 \geq 2$ and $r_2 \geq 3$ and if $r_1 = 2$ then $r_2 = 3$, $\Delta = 6$ and $\omega(\mathcal{H}(G)) \leq 1$.\newline

First assume $x$ is high. Then $d(x) = r_1 + r_2 + 1$ and hence $d_{Z_i(x)}(x) = r_i$ for some $i \in \{1,2\}$. Hence, by Lemma \ref{mozlemma}, either $Z_i(x)$ is complete or is an odd cycle with at least $5$ vertices. In the first case, $Z_i(x)$  contains at least $r_i - \left \lfloor \frac{\Delta(G}{2} \right \rfloor + 2 \geq i \geq 1$.  In the second case, $r_i = r_1 = 2$, so $\mathcal{H}(G)$ is independent.  Thus $Z_i(x)$ contains at least $3$ low vertices.  Hence we can swap $x$ with a low vertex in $U_i$ to get another minimal $\chi(G)$ coloring.\newline

Thus we may assume that $x$ is low.  For $i \geq 0$, let $p_i = 1$ if $i$ is odd and $p_i = 2$ if $i$ is even. Consider the following algorithm.
\begin{enumerate}
\item Put $q_0(y) = 0$ for each $y \in V(G)$.
\item Put $x_0 = x$, $\pi_0 = \pi$ and $i = 0$.
\item Pick a low vertex $x_{i + 1} \in Z_{\pi_i, p_i}(x_i) - x_i$ first minimizing $q_i(x_{i + 1})$ and then minimizing $d(x_i, x_{i + 1})$. Swap $x_{i + 1}$ with $x_i$. Let $\pi_{i+1}$ be the resulting coloring.
\item Put $q_i(x_i) = q_i(x_{i+1}) + 1$.
\item Put $q_{i+1} = q_i$.
\item Put $i = i + 1$.
\item Goto (3).
\end{enumerate}

Since $V(G)$ is finite, we have a smallest $k$ such that we are at step (3),
$p_k = 2$, and $q_k(z) = 1$ for some low vertex $z \in Z_{\pi_k, 2}(x_k) - x_k$.\newline

\textbf{Claim: $q_k(y) \leq 1$ for all $y \in V(G)$.}

Assume to the contrary that we have $y \in V(G)$ with $q_k(y) > 1$, then there is a first $j < k$ for which $q_j(y) > 1$.  From the first minimality condition in step (3) we see that we must have $q_j(t) = 1$ for each low vertex $t \in Z_{\pi_j, p_j}(x_j) - x_j$.  In addition, $p_j = 1$ by the minimality of $k$.\newline

For each low $t \in Z_{\pi_j, p_j}(x_j) - x_j$, let $m(t)$ be the least $a$ such that $t = x_a$.  We will show that there exists low $t \in Z_{\pi_j, p_j}(x_j) - x_j$ such that $x_{m(t)}$ is adjacent to $x_{m(t) + 1}$.  Plainly, this is the case if $r_1 \geq 3$ since then $Z_{\pi_j, p_j}(x_j)$ is complete for all $j$ and $x_{m(t)}$ is always adjacent to $x_{m(t) + 1}$.  Thus we may assume that $r_1 = 2$, $r_2 = 3$, $\Delta = 6$ and $\mathcal{H}(G)$ is independent.  Let $t_1, t_2, \ldots, t_b$ be the low vertices of $Z_{\pi_j, p_j}(x_j)$ ordered by $m(t_l)$.  Since $ Z_{\pi_{m(t_1)}, 1}(t_1)$ is an odd cycle and $\mathcal{H}(G)$ is independent, $Z_{\pi_{m(t_1)}, 1}(t_1)$ contains a pair of adjacent low vertices, say $u$ and $v$. If $N(t_1) \cap Z_{\pi_{m(t_1)}, 1}(t_1)$ contains a low vertex, then $t_1$ is our desired $t$ by the second minimality condition in step (3).  Thus $t_1 \not \in \{u, v\}$.  Take $l$ minimal such that $u = x_{m(t_l) + 1}$ or $v = x_{m(t_l) + 1}$.  Without loss of generality, say $u = x_{m(t_l) + 1}$. Then $t_{l+1}$ must be adjacent to $v$ and thus $t_{l+1}$ is our desired $t$ by the second minimality condition in step (3).\newline

Now, put $a = m(t)$, $H_a = N(x_a) \cap \pi_a^{-1}(C_2)$ and $H_j = N(x_a) \cap \pi_j^{-1}(C_2)$.  Since $x_{a-1} \in H_a$ and $q_{a-1}(x_{a-1}) = 1$, by the minimality of $k$, $N(x_m) \cap H_a = \emptyset$ for $a \leq m < k$.  Thus $H_a \subseteq H_j$.  Since $x_{a+1}$ is adjacent to $x_a$ we have $x_{a + 1} \in H_j - H_a$ and thus $|H_j| \geq |H_a| + 1 = r_2 + 1$.  But then $d(x_a) \geq r_1 + r_2 + 1 \geq \Delta$ contradicting the fact that $x_a$ is low. This proves the claim.\newline

\bigskip

Now, remember our low vertex $z \in Z_{\pi_k, 2}(x_k) - x_k$  with $q_k(z) = 1$. Let $w \in Z_{\pi_k, 2}(x_k) - \{x_k, z\}$ be a low vertex and let $e$ be minimal such that $x_e = z$.  Consider the change of $\pi_k$ given by swapping $x_k$ with $z$ to get a minimal coloring $\pi'$.  Also consider the change of $\pi_k$ given by swapping $x_k$ with $w$ to get a minimal coloring $\pi''$. Since $q_k(x_{e+1}) \leq 1$, it must be that $x_{e+1} \in Z_{\pi', 1}(z) \cap Z_{\pi '', 1}(w)$ and hence $Z_{\pi ', 1}(z) - z = Z_{\pi'', 1}(w) - w$.  Let $T = V(Z_{\pi ', 1}(z)) - z$, $D = V(Z_{\pi_k, 2}(x_k)),$ and $F = G[T \cup D]$.\newline

Since $G$ is critical, we may $\Delta - 1$ color $G - F$.  Doing so leaves a list assignment $J$ on $F$ where $|J(v)| = d_F(v)$ if $v \in V(F)$ is low and $|J(v)| = d_F(v) - 1$ if $v \in V(F)$ is high.  Assume $x_k$ is not adjacent to $x_{e + 1}$.  Since both are low vertices we have $|J(x_k)| + |J(x_{e+1})| \geq d_F(x_k) + d_F(x_{e+1})$.  Clearly, $d_F(z_k) \geq r_2$.  Also, since $x_{e+1}$ is adjacent to all of $D$ we have $d_F(x_{e+1}) \geq r_2 + r_1 - 1$ if $r_1 \geq 3$ and $d_F(x_{e+1}) \geq r_2$ if $r_1 = 2$.  Note that in both cases, $d_F(x_k) + d_F(x_{e+1}) \geq r_1 + r_2 + 1$.  Since the lists together contain at most $\Delta - 1 = r_1 + r_2$ colors, we have $c \in J(x_k) \cap J(x_{e+1})$.  If we color both $x_k$ and $x_{e+1}$ with $c$ it is easy to complete the coloring to the rest of $F$ by first coloring $F - \{z,w, x_k, x_{e+1}\}$ and then coloring $z$ and $w$.  This is a contradiction, hence $x_k$ is adjacent to $x_{e + 1}$.\newline

First assume $\Delta = 6$.  Then $|T| = 2$, say $T = \{z', x_{e + 1}\}$.  Now $D \cup \{x_{e + 1}$ induces a $K_5$ with at most one high vertex and $z'$ is adjacent to the low vertices $w, z \in D$.  Thus Lemma \ref{NoTwoHitter} gives a contradiction.\newline

Hence we may assume that $\Delta \geq 7$.  Put $C = \{z, w\}$, $A = T - \{x_{e+1}\}$ and $B = D - \{z, w\} \cup \{x_{e+1}\}$ and $F' = F - \{z, w\}$.  Then $A$ and $B$ are cliques that cover $F'$ and $x_{e+1}$ is joined to $A$.  As above we may $\Delta - 1$ color $G - F$.  Doing so leaves a list assignment $J$ on $F$ where $|J(v)| = d_F(v)$ if $v \in V(F)$ is low and $|J(v)| = d_F(v) - 1$ if $v \in V(F)$ is high.  If we can find non-adjacent $y_1, y_2 \in V(F')$ such that $J(y_1) \cap J(y_2) \neq \emptyset$, then after coloring $y_1$ and $y_2$ the same we can easily complete the coloring to the rest $F'$ and then to $F$.  Since $G$ contains no $K_{\Delta}$ we have non-adjacent vertices $y_1 \in A$ and $y_2 \in B$. Let $l(y_1, y_2) = |\{i \mid y_i \text{ is low}\}|$ and $n(y_1) = |N(y_1) \cap V(B)|$.  Since $x_{e+1}$ is joined to $A$, $n(y_1) \geq 1$. We have
\begin{align*}
|L(y_1)| + |L(y_2)| &\geq d_F(y_1) + d_F(y_2) - 2 + l(y_1, y_2)\\
&\geq d_{F'}(y_1) + d_{F'}(y_2) + 2 + l(y_1, y_2)\\
&\geq |A| - 1 + n(y_1) + |B| - 1 + 2 + l(y_1, y_2)\\
&= |A| + |B| + n(y_1) + l(y_1, y_2)\\
&= \Delta - 2 + n(y_1) + l(y_1, y_2)
\end{align*}

Since there are at most $\Delta - 1$ colors in both lists, if $n(y_1) + l(y_1, y_2) \geq 2$ we have $L(y_1) \cap L(y_2) \neq \emptyset$ giving a contradiction.  Whence
$n(y_1) + l(y_1, y_2) \leq 1$, giving $l(y_1, y_2) = 0$ and $n(y_1) = 1$.  But $x_k \in B$ is low, so using $y_2 = x_k$ shows that $x_k$ is joined to $A$.  But then $n(y_1) \geq 2$ for any $y_1 \in A$.  This final contradiction completes the proof.
\end{proof}

\end{document}